\documentclass[11pt,oneside, a4paper]{scrartcl}
\usepackage[english]{babel}
\usepackage[utf8]{inputenc}
\usepackage[T1]{fontenc}
\usepackage{lmodern}
\usepackage{amssymb}
\usepackage{amsmath}
\usepackage{amsthm}
\usepackage{amsfonts}
\usepackage[pdftex]{graphicx}
\usepackage{enumitem}
\usepackage{url}
\usepackage{epstopdf}
\usepackage{todonotes}
\usepackage{mathtools}
\usepackage[sort&compress,numbers]{natbib}
\usepackage[linesnumbered,ruled,vlined]{algorithm2e}
\usepackage{overpic}
\usepackage[colorlinks=false, pdfborder={0 0 0}]{hyperref}
\hypersetup{
  pdfauthor={Johannes Siegele, Martin Pfurner, Hans-Peter Schröcker},
  pdftitle={Factorization of Dual Quaternion Polynomials Without Study's Condition},
  pdfkeywords={factorization, dual quaternions, dual quaternion polynomials, rational motion, skew polynomial ring, construction of mechanisms}
  hidelinks,
}

\SetKwInput{KwInput}{Input} 
\SetKwInput{KwReturn}{Return}

\newtheoremstyle{TheoStyle}
    {6pt}                   {8pt}                   {\leftskip=.3cm}                      {-.3cm}                      {\bfseries\scshape}     {:\newline}             { }                     {}                      \theoremstyle{TheoStyle}

\newtheorem{theo}{Theorem}[section]

\newtheorem{lemma}[theo]{Lemma}
\newtheorem{cor}[theo]{Corollary}

\newtheorem{example}[theo]{Example}

\newcommand{\R}{\mathbb R}

\newcommand{\C}{\mathbb C}
\newcommand{\Q}{\mathbb H}
\newcommand{\D}{\mathbb D}
\newcommand{\DQ}{\mathbb{DH}}

\newcommand{\Pj}{\mathbb{P}}
\newcommand{\qi}{\mathbf{i}}
\newcommand{\qj}{\mathbf{j}}
\newcommand{\qk}{\mathbf{k}}
\newcommand{\e}{\varepsilon}
\newcommand{\Cj}[1]{{#1}^\ast}

\newcommand{\norm}[1]{\|#1\|}
\newcommand{\SE}{\operatorname{SE}(3)}

\DeclareMathOperator{\mrpf}{mrpf}

\DeclareMathOperator{\IM}{Im}

\setkomafont{sectioning}{\normalfont\bfseries}

\begin{document}
\title{Factorization of Dual Quaternion Polynomials Without Study's Condition}
\author{Johannes Siegele \and Martin Pfurner \and Hans-Peter Schröcker}

\maketitle

\begin{abstract}
  In this paper we investigate factorizations of polynomials over the ring of
  dual quaternions into linear factors. While earlier results assume that the
  norm polynomial is real (``motion polynomials''), we only require the absence
  of real polynomial factors in the primal part and factorizability of the norm
  polynomial over the dual numbers into monic quadratic factors. This obviously
  necessary condition is also sufficient for existence of factorizations. We
  present an algorithm to compute factorizations of these polynomials and use it
  for new constructions of mechanisms which cannot be obtained by existing
  factorization algorithms for motion polynomials. While they produce mechanisms
  with rotational or translational joints, our approach yields mechanisms
  consisting of ``vertical Darboux joints''. They exhibit mechanical
  deficiencies so that we explore ways to replace them by cylindrical joints
  while keeping the overall mechanism sufficiently constrained.
\end{abstract}

\section{Introduction}

Factorization of dual quaternion polynomials is a powerful tool for the
systematic construction of mechanisms which can perform a given rational motion
\cite{hegedus13,hegedus15}. In \cite{juettler93} it has been shown that rational
motions can be parametrized by motion polynomials, which are defined as
polynomials over the ring of dual quaternions with real norm polynomial (Study's
condition). Factorization algorithms for motion polynomials are presented in
\cite{hegedus13,li15b}. They decompose a given polynomial into products of
linear motion polynomials. The obtained factors parametrize rotations or
translations. Factorizations are generically non-unique. This allows for the
construction of mechanisms with multiple ``legs'', each corresponding to one
factorization. One leg consists of revolute or prismatic joints, obtained from
linear factors in one factorization.

In this paper, we will generalize the factorization results and the algorithm of
\cite{hegedus13} from motion polynomials to dual quaternion polynomials which no
longer have to fulfill Study's condition
(Sections~\ref{sec:fundamental-results}--\ref{sec:translational-factors}). An
obvious necessary condition for a factorization to exist is factorizability of
the norm polynomial over the dual numbers. Moreover, we adopt the genericity
assumption of \cite{hegedus13} and generally require that the primal part has no
real polynomial factor of positive degree.

The factorization results in \cite{hegedus13} are proven for generic rational
motions in this sense and it is known that their trajectories (orbits of points)
are ``entirely circular'', i.\,e. all intersection points with the plane at
infinity have norm zero \cite{li16}. We will show that this no longer needs to
be true for non-motion polynomials. Nonetheless, we prove a necessary
factorizability condition that is related to the intersection points of
trajectories with the plane at infinity and their norms
(Section~\ref{sec:circularity}).

It is well-known that also non-motion polynomials can be used to parametrize
rational motions (cf. for example
\cite{purwar05,purwar10,pfurner16:_path_planning}). In this more general
setting, linear polynomials no longer parametrize just rotations or
translations, they parametrize ``vertical Darboux motions'', cf.
\cite[p.~321]{bottema90} and \cite{purwar10}. These are rotations around a fixed
axis coupled with a harmonic oscillation along this axis. Rotations and
translations are special cases and correspond to amplitude zero or infinity,
respectively, of the harmonic oscillation. This allows us to construct
mechanisms consisting of revolute, prismatic and ``vertical Darboux joints'',
which we show in Section~\ref{sec:mechanisms}.

The caveat of this approach is that no convenient mechanical realizations are
known for vertical Darboux joints. One way to circumvent this problem is to use
cylindrical joints that allow for an independent rotation around and translation
along a fixed axis. They can replace vertical Darboux joints provided the
mechanism still remains sufficiently constrained. In
Sections~\ref{sec:bennett-motions} and \ref{sec:quadratic-polynomials}, we
illustrate this at hand of quadratic polynomials where the mechanism
construction via factorization into \emph{motion} polynomials fails. Our results
yield mechanisms consisting of two revolute and two cylindrical joints. We
extend this construction to certain quadratic dual quaternion polynomials that
do not satisfy Study's condition.

\section{Preliminaries}\label{sec:1}

Let us define the commutative, unital ring $\D=\R[\e]\slash\langle \e^2 \rangle$
of \emph{dual numbers}. Elements of $\D$ can be written as $a+\e b$ with $a$,
$b\in\R$. If $a\neq 0$, the dual number is invertible and its inverse is given
by $(a-\e b)a^{-2}$. The $\D$-algebra generated by the base elements $1$, $\qi$,
$\qj$ and $\qk$ is called the algebra of \emph{dual quaternions}. The
non-commutative multiplication of dual quaternions abides by the rules
\begin{align*}
  \qi^2=\qj^2=\qk^2=\qi\qj\qk=-1,\qquad
  \qi\e=\e\qi,\qquad
  \qj\e=\e\qj,\qquad
  \qk\e=\e\qk.
\end{align*}
The \emph{dual quaternion conjugate} of $q=q_0+q_1\qi+q_2\qj+q_3\qk$ is given by
$\Cj{q}=q_0-q_1\qi-q_2\qj-q_3\qk$. Further we will call the dual number
$\norm{q}=q\Cj{q}=\Cj{q}q$ the \emph{dual quaternion norm} of $q$ (despite the
fact that this is not a norm in the usual sense). A dual quaternion can always
be written as $q=p+\e d$, where $p$ and $d$ are elements of the real associative
algebra $\Q$ of (Hamiltionian) quaternions generated by $(1,\qi,\qj,\qk)$. We
will call $p$ the primal part and $d$ the dual part of $q$. The set of
invertible elements of $\DQ$ will be denoted by $\DQ^{\times}$. Its elements are
all of the form $q=p+\e d$ with $\norm{p} \neq 0$. In this case $\norm{q}$ is
invertible as well and we have $q^{-1}=\Cj{q}\norm{q}^{-1}$.

It is well known, that the set of dual quaternions with real norm modulo real
scalars is isomorphic to the group $\SE$ of rigid body displacements in
three-space \cite{selig05}. A dual quaternion $q=p+\e d$ has the norm
$\norm{q}=\norm{p}+\e(p\Cj{d}+d\Cj{p})$ which is real precisely if
$p\Cj{d}+d\Cj{p}=0$. This is called \emph{Study's condition}. We can think of
this condition in the following way: the map $(p,d)\mapsto 1/2(p\Cj{d}+d\Cj{p})$
is a symmetric bilinear form on $\Q^2$. The dual quaternion $q=p+\e d$ fulfills
Study's condition if and only if the quaternions $p$ and $d$ are orthogonal with
respect to this bilinear form.

Dual quaternions $q=p+\e d$ with $p\neq 0$
fulfilling Study's condition act on a point $(x_1,x_2,x_3)\in\R^3$ by dual
quaternion multiplication via
\begin{align}\label{eq:act}
  x\mapsto \frac{1}{\norm{p}}(p-\e d)x(\Cj{p}+\e\Cj{d})=\frac{1}{\norm{p}}\left(px\Cj{p}+\e(p\Cj{d}-d\Cj{p})\right),
\end{align}
where $x=1+\e(x_1\qi+x_2\qj+x_3\qk)$. The first part of this map, namely
$x\mapsto px\Cj{p}/\norm{p}$ is a rotation given by the quaternion $p$ while the
second part adds the translation vector $(p\Cj{d}-d\Cj{p})/\norm{p}$. For fixed
$p$, the map $d\mapsto p\Cj{d}-d\Cj{p}$ is surjective onto the three-space
spanned by $\qi$, $\qj$ and $\qk$ and its kernel is spanned by $p$. Restricting the
domain of this map to the orthogonal complement of $p$, i.e. to all $d$ such
that $p+\e d$ fulfills Study's condition, we obtain a vector-space isomorphism.
This shows that the group of dual quaternions which fulfill Study's condition
modulo the real multiplicative group $\R^\times$ is isomorphic to $\SE$.

The map in \eqref{eq:act} however is well defined for all
invertible dual quaternions, thus it can be extended to a surjective
homomorphism between $\DQ^\times$ and $\SE$. Two dual quaternions $q$ and $h$
represent the same rigid body displacement if and only if there exists a dual
number $a+\e b$ with $a\neq0$ such that $q=(a+\e b)h$
\cite{purwar05,purwar10,pfurner16:_path_planning}. In fact, for every dual quaternion $q=p+\e d$ with $p\neq 0$, there exists a unique (up to real scalars) dual quaternion, which represents the same displacement and fulfills Study's condition. It can be computed by multiplying the polynomial with the inverse of its norm polynomial (cf. \cite{pfurner18}): 
\begin{align}\label{eq:fiberproj}
\left(\norm{p}-\frac{\e}{2}(p\Cj{d}+d\Cj{p})\right)(p+\e d) = \norm{p}p+\frac{\e}{2} (d\Cj{p}-p\Cj{d})p.
\end{align}

\subsection{Rational Motions and Dual Quaternion Polynomials}

One-parametric rigid body motions are maps $U\to\SE$ from an interval $U
\subseteq \R$ into the group of rigid body displacements $\SE$. We may think of
them as curves in $\SE$. A motion is called \emph{rational,} if all trajectories
are rational curves in $\R^3$. Rational motions can be represented by
polynomials with dual quaternion coefficients fulfilling Study's condition
\cite{juettler93}.

We consider polynomials $\sum_{i=0}^n m_it^i$ with dual quaternion coefficients
$m_0$, $m_1$, \ldots, $m_n \in \DQ$. The non-commutative multiplication of
polynomials of this type is defined by the convention that the indeterminate $t$
commutes with all coefficients. This turns the set of polynomials over the dual
quaternions into a ring which we denote by $\DQ[t]$.

Similar to the notions above we can define the conjugate of a dual quaternion
polynomial $M = \sum_{i=0}^nm_it^i$ as $\Cj{M}=\sum_{i=0}^n\Cj{m_i}t^i$ and the
norm polynomial as $\norm{M}=M\Cj{M}=\Cj{M}M$. It is a polynomial with
coefficients in the dual numbers $\D$. Since dual quaternion multiplication is
non-commutative, polynomials in $\DQ[t]$ do not come with a canonical
evaluation. We will use the so called \emph{right evaluation $M(h)$} of the
polynomial $M$ at $h\in\DQ$. It is obtained by writing the powers of the
indeterminate on the right side of the coefficients before substituting $h$ for
$t$, i.\,e. $M(h) \coloneqq \sum_{i=0}^nm_ih^i$. Dual quaternions for which the
right evaluation of a polynomial $M$ equals zero are called \emph{right zeros}
of $M$. The notions of \emph{left evaluation} and \emph{left zeros} can be
defined by writing the indeterminate at the left side of the coefficients before
substituting. We will occasionally refer to right zeros of $M$ simply as
``zeros''. A polynomial $M\in\DQ[t]$ can always be written as $P+\e D$ where $P$
and $D$ are both polynomials with (Hamiltonian) quaternion coefficients. Again
we call $P$ the primal part and $D$ the dual part of~$M$.

All dual quaternion polynomials with the property that their norm polynomial
does not lie in $\e\R[t]$ parameterize rational motions
\cite{purwar05,purwar10,pfurner16:_path_planning} via \eqref{eq:act}:
\begin{equation*}
  x\mapsto \frac{1}{\norm{P(t)}}(P(t)-\e D(t))x(\Cj{P(t)}+\e\Cj{D(t)}),\quad
  t \in \R.
\end{equation*}
If dual quaternion polynomials are the same up to a real polynomial factor, they
represent the same motions (this is even true for dual polynomial factors). Dual
quaternion polynomials without a real polynomial factor are called
\emph{reduced}, the maximal real polynomial factor of $M\in\DQ[t]$ will be
denoted by $\mrpf(M)$.

Dual quaternion polynomials with invertible leading coefficient and norm
polynomial in $\R[t]$ are called \emph{motion polynomials}
\cite{hegedus13,li15b}.

\subsection{The Vertical Darboux Motion}\label{sec:vdm}

The simplest rational motions are rotations around a fixed axis and translations
in a fixed direction. Both can be represented by linear motion polynomials, i.e.
dual quaternion polynomials of degree one which fulfill Study's condition. Given
a point $(v_1,v_2,v_3)\in\R^3$, the polynomial $t+v\in\DQ[t]$ with
$v=v_1\qi+v_2\qj+v_3\qk$ parametrizes a rotation around the axis with direction
with $v$. Provided $\norm{(v_1,v_2,v_3)}=1$, the rotation angle is given by
$2\cot^{-1}(t)$. The polynomial $t+\e v$ parametrizes a translation by
$-2v/t$. The rotation around an axis parallel to $(v_1,v_2,v_3)$ and through the
point $(x_1,x_2,x_3)$ is given by
\begin{align}\label{eq:rot}
(1-\e x/2)(t+v)(1+\e x/2)=t+v+\e/2(vx-xv),
\end{align}
where $v = v_1\qi + v_2\qj + v_3\qk$ and $x=x_1\qi + x_2\qj + x_3\qk$.

But not all linear polynomials in $\DQ[t]$ describe rotations or translations:
In general, linear non-motion polynomials describe \emph{vertical Darboux
  motions} which consist of a rotation around a fixed axis coupled with a
harmonic oscillation along the same axis such that one period of oscillation
corresponds to one full rotation (cf. \cite[p.~321]{bottema90} for general
information on vertical Darboux motions; the statement about their relation to
linear polynomials can be found for example in \cite{purwar10}). We may view
rotations and translations as special cases of vertical Darboux motions: They
correspond to zero and infinite oscillation amplitude, respectively. General
(not necessarily vertical) Darboux motions have the property
that all trajectories are ellipses (or line segments). For vertical Darboux
motions, these ellipses lie on cylinders around a fixed axis, the axis of the
motion (Figure~\ref{fig:vertical-darboux-motion}), hence they are called
\emph{vertical} (German ``aufrecht'').
\begin{figure}
\centering
  \includegraphics{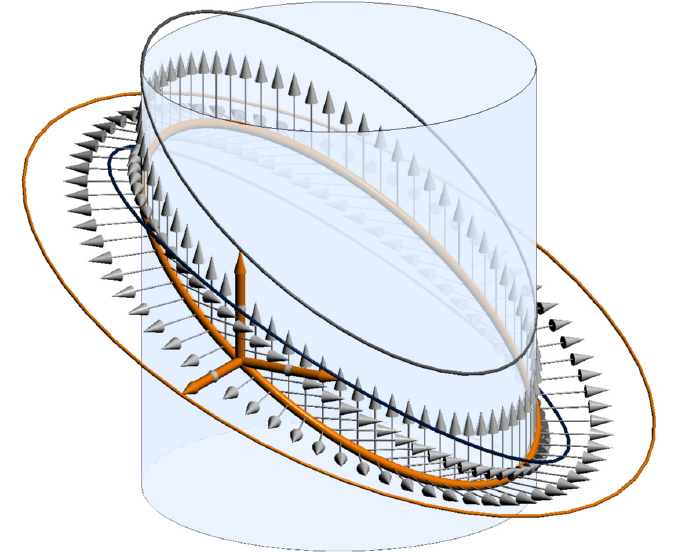}
 \caption{Vertical Darboux motion}
  \label{fig:vertical-darboux-motion}
\end{figure}
A linear (non-motion) polynomial representing a vertical Darboux motion is for
example $t+\qk+\e$. Using \eqref{eq:fiberproj} we can find the motion polynomial
$(t^2+1)(t+\qk)+\e(1-t\qk)$ which represents the same vertical Darboux motion.
This polynomial however is of degree three. Thus, by not restricting ourselves
to motion polynomials, we can represent certain motions by polynomials with
lower degree. In this view, neglecting Study's condition introduces a new class
of elementary motions, namely vertical Darboux motions, as they can be
represented by linear non-motion polynomials.

\section{Factorization of Generic Dual Quaternion Polynomials}\label{sec2}
Factorization of dual quaternion polynomials into linear factors was the topic
of \cite{hegedus13, li15b}. One of their intentions was to use factorization for
the construction of mechanisms, where a linear factor corresponds to a revolute
or prismatic (translational) joint \cite{hegedus15}. The used factorization
algorithms fundamentally rely on the fact that the dual quaternion polynomial $M
= P + \e D$ is a \emph{motion polynomial,} that is, it satisfies the Study
condition $P\Cj{D} + D\Cj{P} = 0$.

In this article, we study the factorization theory of general dual quaternion
polynomials which is interesting in its own right and provides us with
additional possibilities for the construction of mechanisms. The factorization
into linear factors corresponds to a decomposition of the motion into the
product of vertical Darboux motions, coupled by the common parameter $t$, and
with rotations and translations as special cases. Our aim is the generalization
of the results from \cite{hegedus13} to generic polynomials $M \in \DQ[t]$.

\subsection{Fundamental Results}\label{sec:fundamental-results}
In this section we recall known results of \cite{hegedus13} and also present
slightly more general versions of Lemma~3 and Theorem~1 from \cite{hegedus13}.
They are generalized to dual quaternion polynomials which no longer have to
fulfill Study's condition. We give proofs for all new results, even if they are
sometimes very similar to the proofs given in \cite{hegedus13}.

The first lemma concerns polynomial division in a dual quaternion setting. (We
will only use this type of polynomial division in this paper and hence refrain
from using the more explicit term ``polynomial right division''.)
\begin{lemma}[{\cite[Lemma~1]{hegedus13}}]\label{lemma1}
	Let $A$ and $B\in\DQ[t]$ such that $B$ is monic, i.\,e. its leading
  coefficient is $1$. Then there exist unique $Q$, $R\in\DQ[t]$ such that
  $A=QB+R$ and $\deg(R)<\deg(B)$. Further, if $h\in\DQ$ is a root of $B$, then
  $A(h)=R(h)$.
\end{lemma}
Note that this statement can be generalized to a divisor $B$ with invertible
leading coefficient. 

The next lemma states a well-known relation between right zeros and linear right
factors which is also valid for polynomials over more general rings.

\begin{lemma}[{\cite[Lemma~2]{hegedus13}}]
  \label{lemma2}
	Let $M\in\DQ[t]$. A dual quaternion $h\in\DQ$ is a zero of $M$ if and only if
  $t-h\in\DQ[t]$ is a right factor of $M$.
\end{lemma}

If a polynomial $M = P + \e D$ satisfies Study's condition, its norm polynomial
$\norm{M}$ is real. Linear factors of $M$ are closely related to quadratic real
factors of $\norm{M}$. In our setting, the norm polynomial $\norm{M}$ has dual
numbers as coefficients. Its quadratic factors over the dual numbers are also
important for computing linear factors of $M$. This manifests in the following
lemma. Its proof follows the proof of \cite[Lemma~3]{hegedus13} but the more
general assumptions require to consider a few more details.

\begin{lemma}\label{lemma3}
  Let $M=P+\e D\in\DQ[t]$ be a monic polynomial and $N\in \D[t]$ a monic
  polynomial of degree two which is a factor of $\norm{M}$ whose primal part
  does not divide $P$. Then there exists a unique $h\in\DQ$ such that
  $t-h\in\DQ[t]$ is a right factor of $M$ and $\norm{t-h}=N$.
\end{lemma}
\begin{proof}
  By Lemma~\ref{lemma1} there exist $Q$, $R=r_1t+r_0\in\DQ[t]$ such that
  $M=QN+R$. Using this we can see that the norm of $M$ is of the following form:
  \begin{align*}
    \norm{M}=(N\norm{Q}+R\Cj{Q}+Q\Cj{R})N + \norm{R}
  \end{align*}
  The primal part of $N$ does not divide $P$, therefore the primal part of $R$
  cannot be zero. Further, $N$ must be a factor of $\norm{R}$ since it is a
  factor of $\norm{M}$. Thus there exists $\lambda\in\D$ such that
  $\norm{R}=\lambda N$. The scalar $\lambda$ is even invertible as otherwise the
  primal part of $R$ would vanish. The leading coefficient of $\norm{R}$ is
  $\norm{r_1}$, which has a non-zero primal part due to the facts that $N$ is
  monic and $\lambda$ is invertible. This shows that $r_1$ is invertible and
  consequently $h:=r_1^{-1}r_0$ is the unique zero of $R$. By
  Lemma~\ref{lemma2}, there exists $\widetilde{r}\in\DQ$ such that
  $R=\widetilde{r}(t-h)$ and further $\lambda
  N=\norm{R}=(t-\Cj{h})\norm{\widetilde{r}}(t-h)$. Since $\lambda$ is a dual
  number, $t-h$ is also a right factor of $N$ and consequently of $M$.
  Monicity of $N$ implies $N = \norm{t-h}$. This shows existence.

  To prove uniqueness let us assume there exists $\widetilde{h}$ such that
  $M(\widetilde{h})=N(\widetilde{h})=0$. Then $R(\widetilde{h})=0$,
  but the zero of $R$ is unique.
\end{proof}

We will see below that factorizability of $M=P+\e D$ implies that the norm
polynomial $\norm{M}$ must factor into quadratic polynomial factors over the
dual numbers. Now we have all the tools to prove that this condition is also
sufficient, provided $P$ does not have a real polynomial factor of positive
degree. This distinguishes factorization of general polynomials $M \in \DQ[t]$
from the factorization of motion polynomials as investigated in
\cite{hegedus13,li15b}: The latter unconditionally factorize if the
primal part is free from non-trivial real factors since their norm polynomial
always factors into quadratic polynomials.

Let us start by investigating the factorizability of polynomials in $\D[t]$.

\begin{lemma}\label{lem:dualfac}
  Let $p = f+\e g\in\D[t]$ be a monic polynomial. Further let us decompose the
  primal part of $p$ into $f=\prod_{i=1}^mN_i^{n_i}$ where $N_1$, \ldots,
  $N_m\in\R[t]$ are coprime irreducible monic polynomials and $n_1$, \ldots,
  $n_m\in\mathbb{N}$ are positive integers. The polynomial $p$ admits a
  factorization such that all factors are monic and have an irreducible primal
  part if and only if $\prod_{i=1}^mN_i^{n_i-1}$ is a factor of $g$.
\end{lemma}
\begin{proof}
  Let us first assume that $p$ admits a factorization $\prod_{j=1}^n (f_j+\e
  g_j)$ where the $f_j$ are monic and irreducible (not necessarily coprime)
  factors of $f$. Obviously, we get $f=\prod_{j=1}^n f_j$ and further
  \begin{equation*}
    f+\e g=f+\e\sum_{i=1}^ng_i\prod_{j\neq i}f_j,
  \end{equation*}
  which shows that $\prod_{i=1}^mN_i^{n_i-1}$ is a factor of~$g$.

  Let us assume on the other hand that $\prod_{i=1}^mN_i^{n_i-1}$ is a factor of
  $g$. Then $p=\prod_{i=1}^mN_i^{n_i-1}\widetilde{p}$ where
  $\widetilde{p}=\prod_{i=1}^m N_i + \e \lambda$ is a monic polynomial, i.e.
  $\deg(\lambda)\le\deg(\widetilde{p}) -1$. Set $B \coloneqq \prod_{i=1}^m N_i$
  and, for $i \in \{1,\ldots,m\}$, $B_i \coloneqq \prod_{j\neq i}N_j$. The
  polynomial $\widetilde{p}$ admits a factorization of postulated shape, if
  there exist $\lambda_i\in\R[t]$ with $\deg(\lambda_i)<\deg(N_i)$ for $i=1$,
  \ldots, $m$ such that
  \begin{equation*}
    \widetilde{p}=\prod_{i=1}^m (N_i+\e \lambda_i)=B+\e\sum_{i=1}^m\lambda_iB_i.
  \end{equation*}
  Without loss of generality we may assume that there is $k\in\mathbb{N}$ such
  that $N_1$, \ldots, $N_k$ are quadratic polynomials, and $N_{k+1}$, \ldots,
  $N_m$ are linear. To show the existence of such $\lambda_i$, it is sufficient
  to show that the set of polynomials $B_i$, for $i=1$, \ldots, $m$ and $tB_i$
  for $i=1$, \ldots, $k$ form an $\R$-basis of the real vector space of
  polynomials up to degree $\deg(\widetilde{p})-1=m+k-1$. Since the cardinality
  of this set is $m+k$, we only need to show linear independence. Let
  us take $\mu_1,\ldots,\mu_m\in\R$ and $\nu_1,\ldots,\nu_k\in\R$ such that
  \begin{align}\label{eq:1}
    \sum_{i=1}^m \mu_iB_i + \sum_{i=1}^k \nu_itB_i=0.
  \end{align}
  Since all $N_i$ are coprime, they have distinct complex roots $z_i$,
  $\overline{z_i}\in\C\backslash\R$ for $i=1$, \ldots, $k$ and $z_i\in\R$ for
  $i=k+1$, \ldots, $m$. The polynomials $B_i = \prod_{j\neq i}^m N_j$ evaluated
  at $z_\ell$ or $\overline{z_\ell}$ are all zero, except for $i=\ell$. Thus,
  evaluating the left hand side of \eqref{eq:1} at these zeros results in a
  system of equations
  \begin{equation}
    \label{eq:2}
    \begin{aligned}
      \mu_i + \nu_iz_i&=0,\\
      \mu_i + \nu_i\overline{z_i}&=0,   
    \end{aligned}
  \end{equation}
  for $i=1,\ldots,k$ and $\mu_i=0$ for $i > k$. The difference of these two
  equations yields
  \begin{equation*}
    2\nu_i\IM{z_i} = 0.
  \end{equation*}
  Since $z_i \in \C\setminus\R$, $\IM{z_i} \neq 0$ so that $\nu_i = 0$ and hence
  also $\mu_i = 0$ follows. Thus, the set of equations \eqref{eq:2} only has the
  solution $\mu_i=\nu_i=0$ for $i=1,\ldots,k$. Therefore, the polynomials $B_i$
  and $tB_i$ are linearly independent and $p$ consequently admits a
  factorization of the desired shape.
\end{proof}

\begin{theo}\label{theo1}
	Let $M=P+\e D\in \DQ[t]$ be a polynomial with no real polynomial factor in the
  primal part ($\mrpf(P)=1$) and $\norm{P}=\prod_{i=1}^m N_i^{n_i}$ for
  irreducible, quadratic and coprime real polynomials $N_1,\ldots,N_m \in \R[t]$
  and positive integers $n_1,\ldots,n_m \in\mathbb{N}$. The polynomial $M$
  admits a factorization if and only if $\prod_{i=1}^m N_i^{n_i-1}$ is a factor
  of $\norm{M}$. (This is the case if and only if $\norm{M}$ factorizes in the
  sense of Lemma~\ref{lem:dualfac}).
\end{theo}
\begin{proof}
  Let us assume $M=\prod_{j=1}^k F_j$, where $F_j=P_j+\e D_j \in\DQ[t]$ are
  linear polynomials. Then
  \begin{align*}
    \norm{M}=\prod_{j=1}^k \norm{F_j}
    = \prod_{i=1}^mN_i^{n_i} + \e
    \sum_{j=1}^k (P_j\Cj{D_j}+D_j\Cj{P_j}) \norm{P_j}^{-1} \prod_{i=1}^m N_i^{n_i}.
  \end{align*}
  Since for every $j=1,\ldots,k$ there exists $i\in\{1,\ldots,m\}$ such that
  $\norm{P_j}= N_i$, we can conclude that $\prod_{i=1}^m N_i^{n_i-1}$ is a
  factor of $\norm{M}$.

  Let us assume on the other hand that $\prod_{i=1}^m N_i^{n_i-1}$ is a factor
  of $\norm{M}$. By Lemma~\ref{lem:dualfac} we know that $\norm{M}$ admits a
  factorization into quadratic and monic polynomials $q_1$, $\ldots$, $q_k$ in
  $\D[t]$. By assumption $P$ does not have a real factor. Thus we can use
  Lemma~\ref{lemma3} to find $h_k\in\DQ$ and $Q\in\DQ[t]$ such that $M=
  Q(t-h_k)$ and $\norm{t-h_k}=q_k$. Obviously, $Q$ cannot have a real
  polynomial factor in the primal part as otherwise $M$ would have a factor in
  the primal part which contradicts our assumptions. Further
  $\norm{M}=\norm{Q}q_k$, whence $\norm{Q}=\prod_{i=1}^{k-1}q_i$. Thus we can
  use Lemma~\ref{lemma3} recursively to obtain $h_1, \ldots, h_{k-1}\in\DQ[t]$
  such that $M=\prod_{i=1}^k (t-h_i)$ and $\norm{t-h_i}=q_i$.
\end{proof}

\begin{example}\label{ex:2.7}
  The polynomial $M=P+\e D$ with $P=(t-\qi)(t-\qk)$ and $D=t-\qj$ does not admit
  a factorization. The norm of $M$ equals $(t^2+1)^2+2\e (t^3+1)$ but its dual
  part does not have the factor $t^2+1$. This violates the necessary
  factorization condition of Theorem~\ref{theo1}.
\end{example}

\subsection{Algorithm}
\label{sec:algorithm}

The proofs of Lemma~\ref{lemma3} and Theorem~\ref{theo1} are constructive and
similar to the proofs in \cite{hegedus13}. The difference is, that we need to
use quadratic \emph{dual} polynomial factors of the norm polynomial. Nonetheless
the same Algorithm~\ref{alg1} can be used to compute factorizations.

\begin{algorithm}[H]\label{alg1}
  \caption{Rfactor {\cite[Algorithm~1]{hegedus13}}}
  \KwInput{$M=P+\e D\in\DQ[t]$ such that $\mrpf(P)=1$ and $\norm{M}$ admits a
    factorization in the sense of Lemma~\ref{lem:dualfac}.}
  \If{$\deg(M)=1$}{
    \KwReturn{M.}
  }
  $N\leftarrow$ monic, quadratic factor of $\norm{M}$ in $\D[t]$ \\
  $h\leftarrow$ common dual quaternion zero of $N$ and $M$\\
$M \leftarrow$ quotient of right-division of $M$ by $t-h$\\
    \KwReturn{$\operatorname{Rfactor}(M),(t-h)$.}
\end{algorithm}

Algorithm~\ref{alg1} is not deterministic as its output depends on the choice of
the quadratic factor $N$ in Line~2. In the generic case, the norm polynomial of
$P$ has $n$ coprime irreducible quadratic factors in $\R[t]$. Then the
factorization of $\norm{M}$ is unique up to permutations of the factors, as we
have seen in the proof of Lemma~\ref{lem:dualfac}. In this case, there exist
$n!$ different factorizations of $M$, as in the original case of ``generic
motion polynomials'' \cite{hegedus13}. The situation is different, if the norm
of $P$ has quadratic factors of higher multiplicity.

\begin{cor}
  \label{cor:inffac}
  Let $M=P+\e D\in\DQ[t]$ be a dual quaternion polynomial that admits a
  factorization. If the norm of $P$ has a quadratic factor $N \in \R[t]$ with
  multiplicity $m \ge 2$, then $M$ admits infinitely many factorizations.
\end{cor}

\begin{proof}
  The norm polynomial $\norm{M}$ has the factor $N^m+\e N^{m-1}\lambda$ for a
  linear polynomial $\lambda\in\R[t]$. There are infinitely many ways to write
  $\lambda$ as a sum of $m$ linear polynomials $\lambda_i$ so that
  \begin{align*}
    N^m+\e N^{m-1}\lambda = N^m+\e N^{m-1}\sum_{i=1}^m
    \lambda_i = \prod_{i=1}^m (N+\e \lambda_i).
  \end{align*}
  Consequently, $\norm{M}$ has infinitely many factorizations and so does~$M$.
\end{proof}

\subsection{Translational Factors}
\label{sec:translational-factors}
Algorithm~\ref{alg1} is known to work for many more general cases (without the
assumption $\mrpf(P) = 1$) but cases of failure are known as well. In
particular, it was already observed in \cite{hegedus13} that the algorithm is
applicable to motion polynomials where the irreducible real polynomial factors
of its primal part are linear and at most of multiplicity one. This is still the
case for general polynomials $M=cP+\e D\in\DQ[t]$, where $c=\prod_{i=1}^n
c_i\in\R[t]$ is a product of distinct linear polynomials $c_i$ and $P$ has no
real polynomial factor. The norm polynomial of $M$ has $c$ as factor and the
primal part of $\norm{M}$ has $c^2$ as factor. Thus we can use
Lemma~\ref{lem:dualfac} to find linear $\lambda_i\in\R[t]$ such that $\norm{M}$
has the factors $c_i^2+\e\lambda_i$. The primal part of these factors do not
divide the primal part of $M$ due to the assumption that every linear factor of
the primal part has multiplicity one, thus we can use Lemma~\ref{lemma3} to
compute linear factors $c_i + \e (d_1\qi + d_2\qj + d_3\qk)$ of $M$ and
consequently, the Algorithm~\ref{alg1} still works in this more general setting.
The linear factors $c_i + \e (d_1\qi + d_2\qj + d_3\qk)$ parameterize all
translations in the fixed direction $(d_1,d_2,d_3)$, hence the name
``translational factors''.

If on the other hand, we have a reduced polynomial $M=c^2P+\e D$ where
$c\in\R[t]$ is a linear polynomial, it does not admit a factorization. To see
this, let us assume the contrary $M=\prod_{i=1}^n (P_i+\e D_i)$. Since the
primal part of $M$ is the product of the primal parts of the factors, we know
that two of the factors have $c$ as primal part. Thus, there exist $Q_1$, $Q_2$,
$F_1$ and $F_2\in\Q[t]$ such that $M=(cQ_1+\e F_1)(cQ_2+\e F_2) = c(cQ_1Q_2 +
\e(Q_1F_2+F_1Q_2))$. But this contradicts the assumption that $M$ is reduced.

\subsection{Circularity and Factorizability}
\label{sec:circularity}

It is known that generic motion polynomials $Q = P + \e
D$ (polynomials in $\DQ[t]$ with $\mrpf(P) = 1$ and $P\Cj{D} + D\Cj{P} = 0$)
have the following properties:
\begin{itemize}
\item Their norm polynomial $\norm{P}$ factors into quadratic irreducible
  polynomials over $\R$,
\item they always admit factorizations and
\item their generic trajectories are entirely circular \cite[Theorem~1]{li16}.
\end{itemize}

The circularity of an algebraic curve is defined as half the number of
intersection points with the absolute circle at infinity counted with their
multiplicities. A curve is entirely circular if all its intersections with the
plane at infinity lie on the absolute circle.

Let us quickly cast these concepts into algebraic equations. In the dual
quaternion formalism, $\Pj^3(\R)$ is identified with the projective space over
the vector space spanned by $1$, $\e\qi$, $\e\qj$, $\e\qk$. Writing a general
vector as $x_0 + \e x$ with $x = x_1\qi + x_2\qj + x_3\qk$, the plane at
infinity is described by $x_0 = 0$ and the absolute circle at infinity is given
by the additional equation $\norm{x} = 0$. If $x_0 + \e x$ is not a constant
dual quaternion but a dual quaternion polynomial (that is, a rational
parameteric curve), the parameter values of its intersection points with the
plane at infinity are the zeros of $x_0$. If they are also zeros of $\norm{x}$,
they contribute, with their respective multiplicity, to the curve's circularity.

Consider now a general polynomial $M = P + \e D$ with $\mrpf(P) = 1$ whose norm
polynomial $\norm{M}$ factors over $\D$ into quadratic polynomials. We have
already seen that $M$ admits a factorization. Now we investigate necessary
properties of its trajectories. It turns out that these are much weaker than in
the motion polynomial case.

\begin{theo}
  \label{thm:circularity}
  Let $M=P+\e D\in\DQ[t]$ be a polynomial such that $\mrpf(P)=1$ and $M$ admits
  a factorization. Then the trajectories of the motion parameterized by $M$ have
  the following property: All intersection points of the curve with the plane at
  infinity with multiplicity $\mu>1$ intersect the absolute circle with
  multiplicity $\mu$.
\end{theo}
\begin{proof}
  The trajectory of an arbitrary point $x_0+\e x$ with $x=x_1\qi+x_2\qj+x_3\qk$
  in projective three-space with respect to $M$ is given by
  \begin{align}\label{eq2}
    (P-\e D)(x_0+\e x)(\Cj{P}+\e\Cj{D})&= x_0\norm{P}+\e(Px\Cj{P}+x_0(P\Cj{D}-D\Cj{P})).
  \end{align}
  This curve's intersection points with the plane at infinity correspond to zeros
  of its primal part and consequently to the factors of $\norm{P}$. Let us write
  $\norm{P}=\prod_{i=1}^m N_i^{n_i}$ where $N_1$, $\ldots$, $N_m\in\R[t]$ are
  coprime irreducible polynomials. We need to show that all $N_i^{n_i}$ with
  $n_i>1$ are factors of the norm of the dual part in \eqref{eq2}. Let
  us have a look at the polynomial
  \begin{multline*}
    \norm{Px\Cj{P}+x_0(P\Cj{D}-D\Cj{P})}=
    \norm{P}^2\norm{x} - x_0\norm{P}(Px\Cj{D}+D\Cj{x}\Cj{P})\\+ x_0P(x\Cj{P}D+\Cj{D}P\Cj{x})\Cj{P} +x_0^2\norm{P\Cj{D}-D\Cj{P}}.
  \end{multline*}
  The first two summands have $\norm{P}$ as a factor. The middle factor of the
  third summand is a real polynomial, thus it commutes with $P$ and therefore
  the third summand also has the factor $\norm{P}$. For the last summand it
  holds
  \begin{align*}
    \norm{P\Cj{D}-D\Cj{P}}&=4\norm{P}\norm{D}-(P\Cj{D}+D\Cj{P})^2.
  \end{align*}
  Since $M$ admits a factorization, we know that $c \coloneqq \prod_{i=1}^m
  N_i^{n_i-1}$ is a factor of the dual part of $\norm{M}$ which is
  $P\Cj{D}+D\Cj{P}$. Thus the last summand has the factor $c^2$ which in turn
  has the factors $N_i^{n_i}$ with $n_i>1$. This proves the statement.
\end{proof}

The converse of Theorem~\ref{thm:circularity} is not true which can be seen
in the following example.
\begin{example}
  Let $M=(t-\qi)(t-\qk)+\e(t-\qj)$ be the polynomial of Example~\ref{ex:2.7}. The
  polynomial $Q=M(t-\qk)^2$ does not admit a factorization, because it does not
  fulfill the requirements of Theorem~\ref{theo1}, but its norm polynomial
  $\norm{Q}$ has the factor $(t^2+1)^2$. From the proof above we can conclude
  that the trajectories of $Q$ are entirely circular even though the polynomial
  does not admit a factorization.
\end{example}

\section{Construction of Mechanisms}
\label{sec:mechanisms}

In \cite{hegedus13,li15b} the authors presented algorithms to find
factorizations of monic motion polynomials $M=P+\e D \in \DQ[t]$ into linear
factors where all factors are motion polynomials themselves, i.\,e. they
parametrize rotations or translations. This composition of rotations and
translations can be used to construct an open kinematic chain with revolute and
prismatic joints, coupled by the common parameter $t$, which can perform the
rational motion given by $M$. Different factorizations yield different open
chains which can be coupled to reduce the degrees of freedom in the resulting
kinematic structure. This is illustrated in Figure~\ref{fig:kinematic-chains}
for a quadratic motion polynomial $M$ with two distinct factorizations
$M=F_1F_2=G_1G_2$. Each factorization yields a kinematic chain consisting of two
revolute joints, which are interlocked to obtain the depicted mechanism. The
connecting link of $F_2$ and $G_2$ performs the desired one-parametric motion
given by $M$.

As is usual in mechanism science, we denote revolute joints by the letter ``R''
and prismatic joints by the letter ``P''. The mechanism of
Figure~\ref{fig:kinematic-chains} would then be referred to as closed RRRR- or
4R-mechanism.

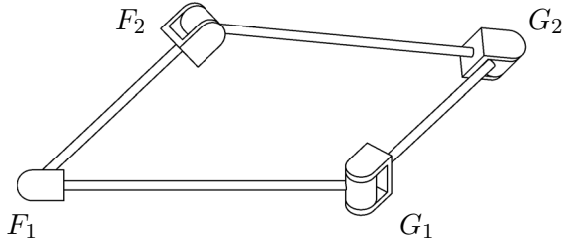
\begin{figure}
  \centering
\begin{overpic}[trim=400 0 0 153,clip,scale=0.5]{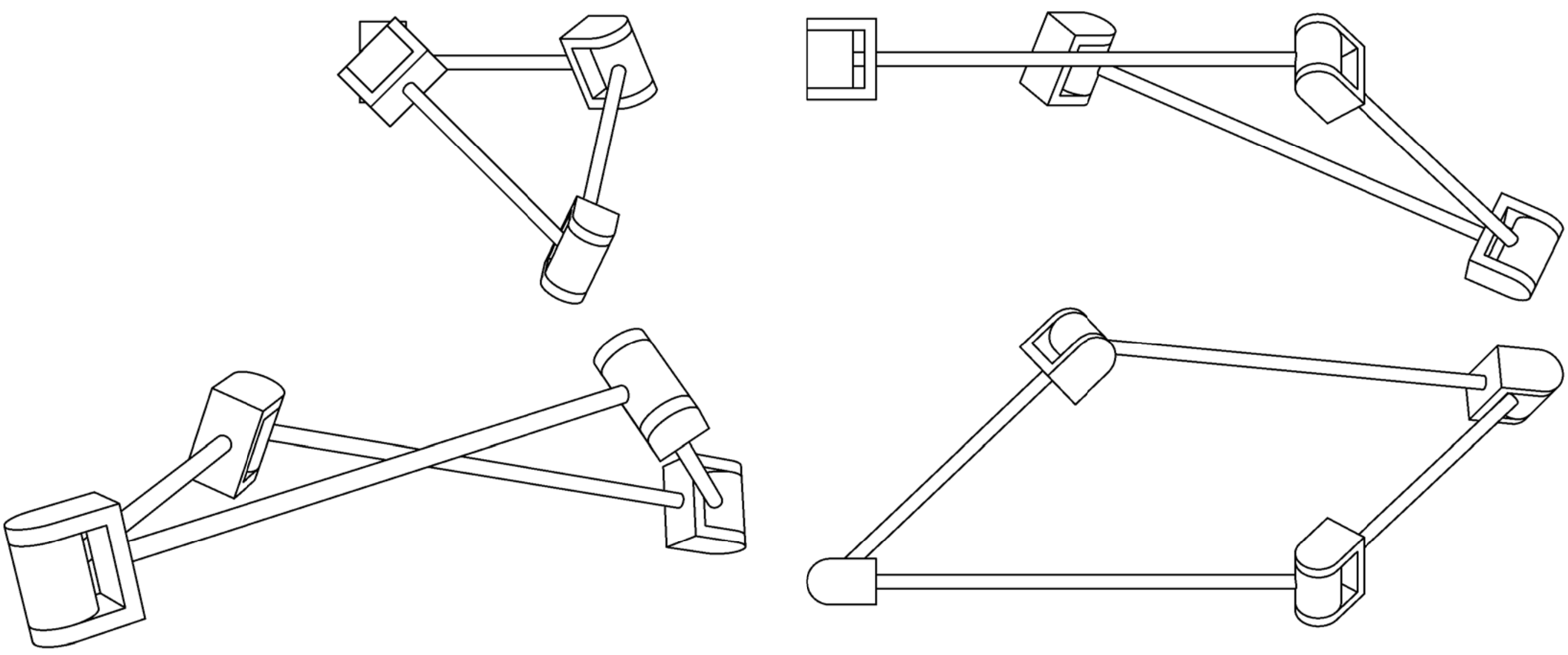}
  \put(-1,0){$F_1$}
  \put(75,0){$G_1$}
  \put(100,40){$G_2$}
  \put(20,40){$F_2$}
\end{overpic}
  \caption{Realization of two factorizations $M = F_1F_2 = G_1G_2$ as a mechanisms.}
  \label{fig:kinematic-chains}
\end{figure}
In this paper we consider polynomials in $\DQ[t]$ which no longer need to
fulfill Study's condition. This results in factorizations where not all linear
factors are motion polynomials and hence represent vertical Darboux motions. In
theory this would allow the construction of kinematic chains with ``vertical
Darboux'' joints (and revolute or prismatic joints in special cases). Darboux
joints are problematic from an engineering viewpoint. Thus we will focus on a
different approach where vertical Darboux joints are simply replaced by
cylindrical (``C'') joints that allow the simultaneous rotation around and
translation along an axis. Of course, this comes with the downside that
cylindrical joints have two degrees of freedom. This will increase the degrees
of freedom of the corresponding open chain but, if it is sufficiently
constrained otherwise, does not increase the degrees of freedom of the complete
mechanical system: The cylindrical joints will actually just perform the
respective vertical Darboux motions. In this way, kinematic chains obtained by
the known algorithms might be coupled with chains obtained by our results to
construct mechanisms with revolute, prismatic and cylindrical joints.

As cylindrical joints introduce an additional degree of freedom, it might be
desirable to keep the number of C-joints low. This can be done by a certain
extend by choosing appropriate quadratic factors of the norm polynomial in
Algorithm~\ref{alg1}. We will use this in
Section~\ref{sec:quadratic-polynomials} to obtain over-constrained mechanisms
corresponding to certain quadratic dual quaternion polynomials which is only
feasible by keeping the number of C-joints minimal.

\subsection{Bennett Motions and Four-Bar Linkage}
\label{sec:bennett-motions}
Quadratic motion polynomials generically represent so-called Bennett motions
\cite{hamann11}. In general, i.\,e. when the norm polynomial is the product of
two coprime real quadratic polynomials, we can use the factorization algorithm
for motion polynomials to obtain two distinct factorizations which in turn
correspond to a closed chain with four revolute joints, a closed 4R-linkage or
``Bennett mechanism'' (Figure~\ref{fig:kinematic-chains}). But if the norm
polynomial is a quadratic irreducible polynomial to the power of two, there
exists only one factorization into linear motion polynomials and the
construction fails. (This is actually a quite common case. The motion can be
obtained by composing two rotations around different axes but with identical
angular velocities.) Our algorithm allows us to find another factorization into
linear (non-motion) polynomials, which enables us to construct an RRCC-linkage
that is capable of performing the given Bennett motion.

\begin{example}
  The motion polynomial $M=t^2-((1+2\e)\qi +\qk + \e\qj)t-\qj+\e(\qi-2)$ has the
  norm polynomial $(t^2+1)^2$. It only admits a unique factorization into motion
  polynomials, namely $M=G_1G_2$ with $G_1=t-\qi-\e\qj$ and $G_2=t-\qk-2\e\qi$.
  But we can also write $\norm{M}=(t^2+1+\e \lambda)(t^2+1-\e\lambda)$ for an
  arbitrary linear polynomial $\lambda=\lambda_1t+\lambda_0\in\R[t]$. Using
  these factors of the norm polynomial in Algorithm~\ref{alg1} yields the
  additional factorization $M=F_1F_2$ where
    \begin{align*}
    F_1 &= G_1 + \frac{\e}{2}((\qi+\qk)\lambda_0-(1+\qj)\lambda_1),\\
    F_2 &= G_2 -\frac{\e}{2}((\qi+\qk)\lambda_0-(1+\qj)\lambda_1).
  \end{align*}
  Note that $F_1$ and $F_2$ are not motion polynomials and therefore parametrize
  ``proper'' vertical Darboux motions. Thus we need to use C-joints to obtain an
  open kinematic chain corresponding to this factorization. We can couple this
  2C-chain with the 2R-chain obtained by the first factorization, since
  $M=F_1F_2=G_1G_2$. To specify the mechanism corresponding to these
  factorizations, it suffices to calculate the angles and distances between the
  joint axes. They are independent from the mechanism's configuration which
  depends on the common joint parameter $t$. In the following, we will calculate
  these values for the mechanism corresponding to the values $\lambda_0=3$,
  $\lambda_1=4$. To find the axes of $G_1$ and $G_2$, we can use \eqref{eq:rot}.
  They are $(0,0,-1)+\R(1,0,0)$ and $(0,-2,0)+\R(0,0,1)$. The axes of $F_1$ and
  $F_2$ can be found by computing the fixed line of these transformations. For
  $F_1$ we know the direction of the axis from the primal part, namely
  $(1,0,0)$, thus it suffices to find a point on the axis. For $t=0$, the image
  of an arbitrary line $(x_1,x_2,x_3)+\R(1,0,0)$ is $(x_1+4,-x_2 - 3,-x_3 - 6) +
  \R(1,0,0)$. Subtracting the point $(x_1,x_2,x_3)$ should yield the line
  $\R(1,0,0)$. We infer $x_2=-3/2$, $x_3=-3$ while $x_1$ can be chosen
  arbitrarily, for example $x_1=0$. The axis of $F_2$ can be found similarly, it
  is given by $(-2,-7/2,0)+\R(0,0,1)$. The angles between $F_1$ and $F_2$, or
  $G_1$ and $G_2$ are $\pi/2$, respectively. The axes of $F_1$ and $G_1$ or
  $F_2$ and $G_2$ are parallel. The distance between $F_1$ and $F_2$ or $G_1$
  and $G_2$ is $2$, the distance between $F_1$, $G_1$ or $F_2$, $G_2$ is $5/2$,
  respectively. Figure~\ref{RRCC} depicts the resulting mechanism.

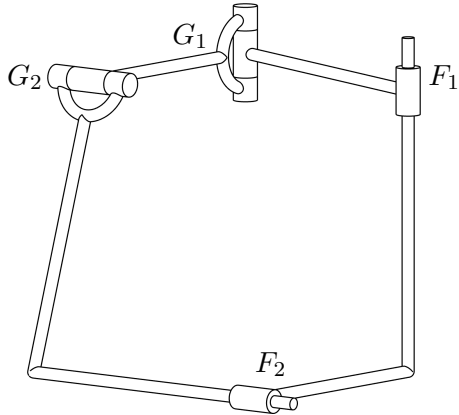
\begin{figure}
  \centering
\begin{overpic}[scale=0.3]{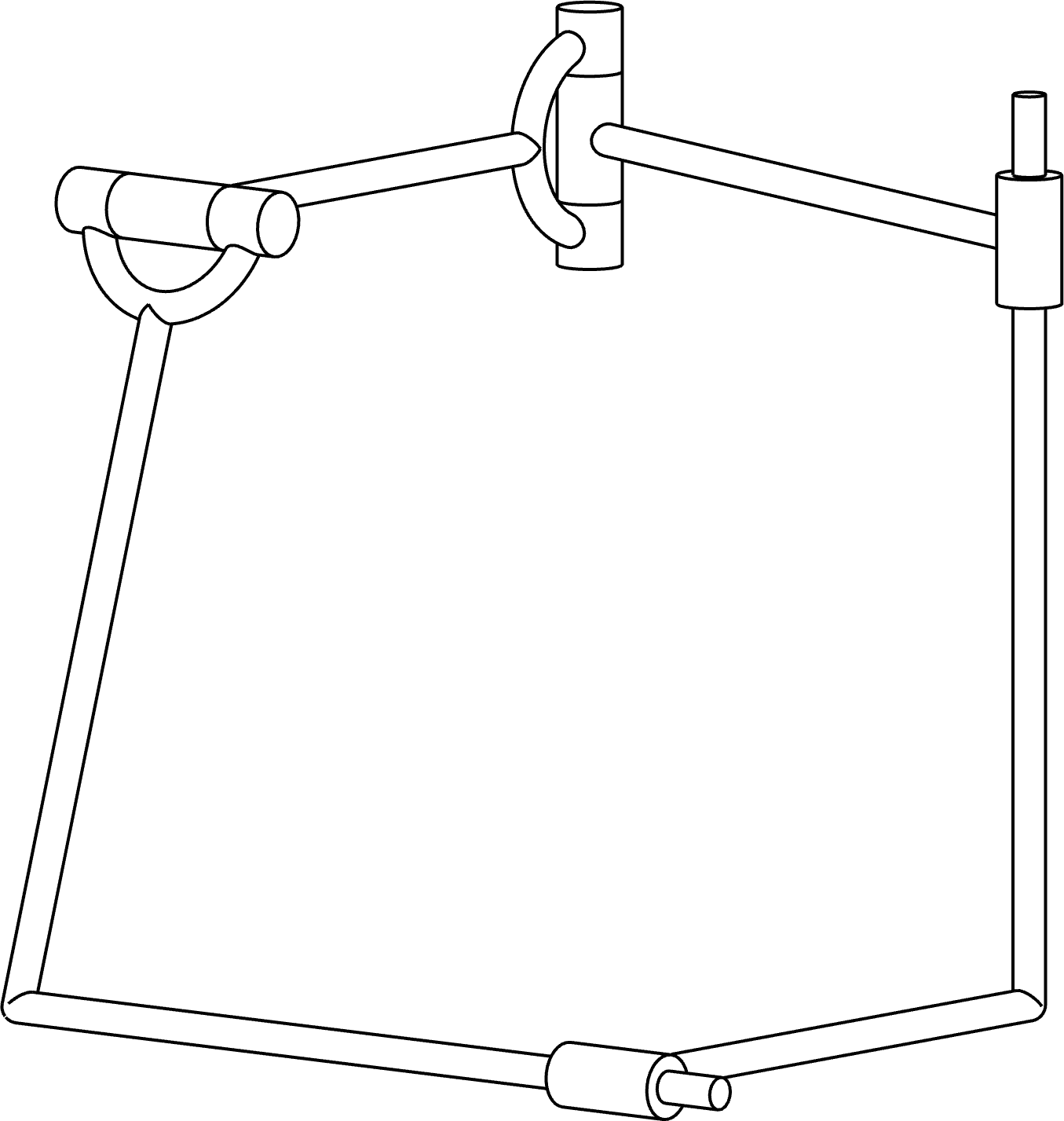}
  \put(35,90){$G_1$}
  \put(55,10){$F_2$}
  \put(97,80){$F_1$}
  \put(-5,80){$G_2$}
\end{overpic}
     \caption{An RRCC-mechanism to perform a Bennett motion.}\label{RRCC}
\end{figure}

\end{example}

\subsection{Quadratic Polynomials and Four-Bar Linkages}
\label{sec:quadratic-polynomials}
While all rational motions can be represented by motion polynomials, the degree
of this representation might not be optimal. Consider a generic polynomial $Q =
P + \e D$. We can use equation~\eqref{eq:fiberproj} to obtain the motion polynomial $M=\norm{P}P+\frac{1}{2}\e (P\Cj{D}-D\Cj{P})P$ which represents the same motion. In general, the
degree of $M$ is three times the degree of $Q$ (it can be lower, if real
polynomials can be canceled). Thus factorizing $Q$ instead of $M$ might result
in fewer factors and therefore in mechanisms with fewer joints. We study this in the case of quadratic polynomials in $\DQ[t]$. Their representations as motion
polynomials are of degree four or six, which results in factorizations with four
or six linear motion polynomials, respectively. If we get six linear factors,
each factorization corresponds to an open 6R-chain which parametrize the whole
$\SE$. Coupling different 6R-chains therefore does not constrain the obtained
mechanism and we cannot obtain a mechanism with a degree of freedom lower than
six. In the case that $M$ has degree four, factorizations of $M$ correspond to
open 4R-chains. Interlocking two different 4R-chains yields a closed
8R-mechanism which has two degrees of freedom. To obtain a mechanism which only
allows for a one-parametric motion we need to couple this closed 8R-mechanism
with a third open 4R-chain. This approach therefore results in a rather
complicated mechanism.
The quadratic polynomial on the other hand admits a factorization into two
linear polynomials, provided it meets the requirements of Theorem~\ref{theo1}.
This allows for the construction of open CC-chains. This is actually just a
special RPRP-chain (revolute and prismatic joints alternate) in disguise and in
this sense does not offer advantages over 4R-chains.

However, if the norm polynomial of the primal part is a square, there exist two
different factorizations, each with a linear motion polynomial as one factor.
Combining the corresponding kinematic chains, an open RC- and an open CR-chain,
gives a closed RCRC four-bar linkage. It has just one degree of freedom and
performs the motion given by~$Q$ and $M$.

\begin{example}
  For $P=t^2-t(\qi+\qk)-\qj$ and $D=(2\qk-4\qi-1)t+(\qi-\qj-1)$ the motion
  polynomial $M=\norm{P}P+\frac{1}{2}\e (P\Cj{D}-D\Cj{P})P$ is of degree four
  (after dividing off a real polynomial factor), but the representation as a
  general polynomial $Q = P + \e D$ is just of degree two. The motion
  polynomial $M$ admits factorizations into four linear motion polynomials and
  thus the construction of open 4R-chains.
    The general polynomial $Q$ on the other hand can be factored as 
  $Q=F_1F_2=G_1G_2$ with
  \begin{gather*}
    F_1 = t-\qi+3\e\qk,\quad
    F_2 = t-\qk-e(1+4\qi+\qk),\\
    G_1 = t-\qi-\e(1+\qi+\qj-2\qk),\quad
    G_2 = t-\qk-\e(3\qi-\qj).
  \end{gather*}
  The first factorization $F_1F_2$ corresponds to an RC-chain, the second
  factorization $G_1G_2$ gives a CR-chain. Combining them yields the RCRC
  four-bar linkage of Figure~\ref{RCRC}.
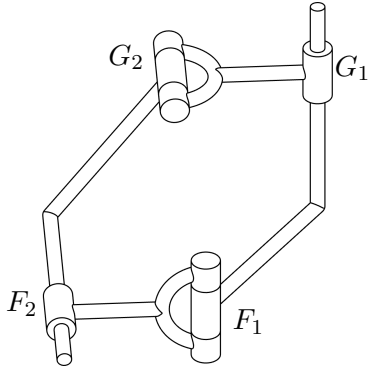
\begin{figure}
  \centering
\begin{overpic}[scale=0.3]{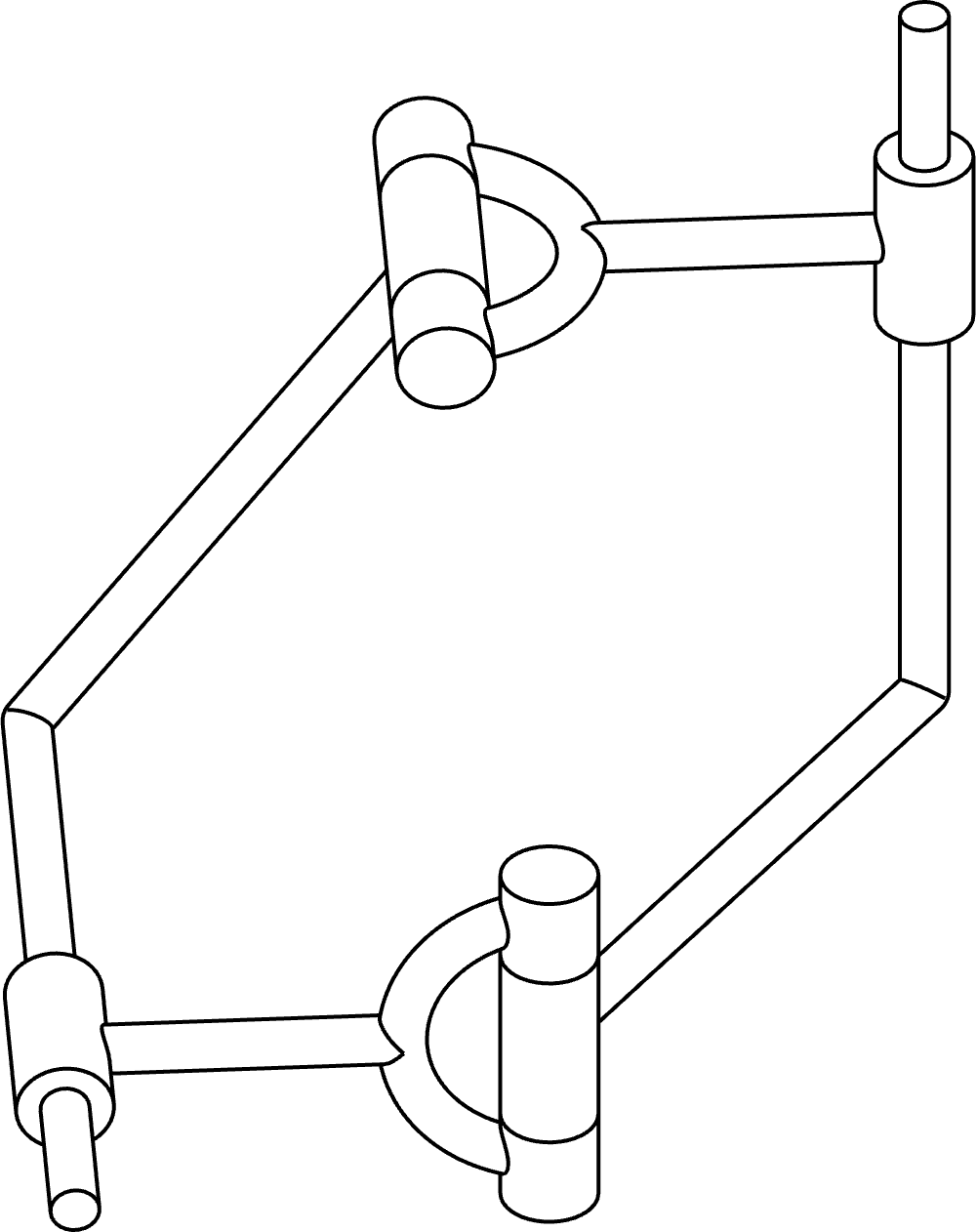}
  \put(18,83){$G_2$}
  \put(80,80){$G_1$}
  \put(52,10){$F_1$}
  \put(-10,15){$F_2$}
\end{overpic}
     \caption{An RCRC-mechanism constructed from a quadratic polynomial.}\label{RCRC}
\end{figure}

\end{example}

\section{Concluding Remarks}
\label{sec:concluding-remarks}

We have shown how to factorize polynomials over the dual quaternions under two
assumptions:
\begin{itemize}
\item The norm polynomial factorizes into quadratic factors over the dual
  numbers and
\item the primal part does not have a real polynomial factor of positive degree.
\end{itemize}
The first condition is obviously necessary for existence of a factorization, the
second condition is a certain restriction. Precise criteria for existence of
factorizations and algorithms for their computation in the excluded non-generic
case are subject of ongoing research. Even in case of motion polynomials a
complete answer is yet unknown.

We consider our factorization results as interesting in their own right but we
also demonstrated that they complement well the known constructions of
mechanisms from motion polynomials. It seems that abandoning Study's condition
(and thus giving up uniqueness of motion representation) is a promising concept
and deserves further investigations.

\section*{Acknowledgments}

Johannes Siegele was supported by the Austrian Science Fund (FWF): P~30673
(Extended Kinematic Mappings and Application to Motion Design).

\bibliographystyle{alpha}

\end{document}